\documentclass[12pt,reqno]{article}

\usepackage[usenames]{color}
\usepackage{amssymb}
\usepackage{amsmath}
\usepackage{amsthm}
\usepackage{amsfonts}
\usepackage{amscd}
\usepackage{graphicx}
\usepackage{tikz}
\usepackage{latexsym}

\usepackage[colorlinks=true,
linkcolor=webgreen,
filecolor=webbrown,
citecolor=webgreen]{hyperref}

\definecolor{webgreen}{rgb}{0,.5,0}
\definecolor{webbrown}{rgb}{.6,0,0}

\usepackage{color}
\usepackage{fullpage}
\usepackage{float}

\newcommand{\seqnum}[1]{\href{https://oeis.org/#1}{\underline{#1}}}

\newcommand{\rci}{\rm rci}

\newcommand{\exc}{{\rm exc}}
\newcommand{\inv}{{\rm inv}}
\newcommand{\crs}{{\rm cr}}
\newcommand{\nes}{{\rm nes}}
\newcommand{\st}{{\rm st}}
\newcommand{\ut}{{\rm ut}}
\newcommand{\Ut}{{\rm Ut}}
\newcommand{\lt}{{\rm lt}}
\newcommand{\Lt}{{\rm Lt}}

\begin{document}
	
\theoremstyle{plain}
\newtheorem{theorem}{Theorem}
\newtheorem{corollary}[theorem]{Corollary}
\newtheorem{lemma}[theorem]{Lemma}
\newtheorem{proposition}[theorem]{Proposition}

\theoremstyle{definition}
\newtheorem{definition}[theorem]{Definition}
\newtheorem{example}[theorem]{Example}
\newtheorem{conjecture}[theorem]{Conjecture}

\theoremstyle{remark}
\newtheorem{remark}[theorem]{Remark}

\begin{center}
	\vskip 1cm{\LARGE\bf  Crossings over permutations avoiding some pairs of patterns of length three\\
		\vskip 1cm}
	\large
	Paul M. Rakotomamonjy\footnote{Corresponding author.},  Sandrataniaina R. Andriantsoa \\ 
	Arthur Randrianarivony\\
	Department of Mathematics and Computer Science\\ 
	Domain of Sciences and Technology\\
	University of Antananarivo\\
	 Madagascar\\ 
	\href{rpaulmazoto@gmail.com}{\tt rpaulmazoto@gmail.com}, 
	\href{andrian.2sandra@gmail.com}{\tt andrian.2sandra@gmail.com}\\
	\href{arthur.randrianarivony@gmail.com}{\tt arthur.randrianarivony@gmail.com}
\end{center}

\vskip .2 in

\begin{abstract} 	
In this paper,  we compute the distributions of the statistic number of crossings over  permutations avoiding one of the pairs $\{321,231\}$, $\{123,132\}$ and $\{123,213\}$. The obtained results are new combinatorial interpretations of two known triangles in terms of restricted permutations statistic. For other pairs of patterns of length three, we find relationships between the polynomial distributions of the crossings over permutations that avoid the pairs containing the pattern 231 on the first hand and the pattern 312 on the other hand.
	\begin{center}
	
\end{center}

\end{abstract}
	
\section{Introduction and main results}\label{sec1} The statistic
number of crossings is among the complicated statistics on
permutations. Its survey arises from the works of de M\'edicis and
Viennot \cite{MedVienot}, Randrianarivony \cite{ARandr1,ARandr},
Corteel \cite{Cort},  Burrill et al.~\cite{Burril} to Corteel
et al.~\cite{Cort2}. Recently, the first author of this paper
introduced the study of this statistic on permutations avoiding a
single pattern of length three \cite{Rakot}. This one is devoted
on the distribution of crossings  on permutations avoiding a
pair of patterns of length three. The technique we use in this
paper differs from that of these known works who generally used
a bijection between permutations and a family of paths. Here,
we simply manipulate the structure of our combinatorial objects
and  use some trivial bijections that we will present in the
next sections.

A permutation $\sigma$ of $[n]:=\{1,2,\ldots,n\}$ is a bijection from $[n]$ to itself that can be written linearly as $\sigma=\sigma(1)\sigma(2)\cdots \sigma(n)$. We shall refer $n$ as the length of $\sigma$ (i.e.,  $n=|\sigma|$) and we let $S_n$ denote the set of all permutations of length $n$.  A \textit{crossing}  of a given permutation $\sigma$ is a pair of indices $(i,j)$ such that $i<j< \sigma(i)<\sigma(j)$  or  $\sigma(i)<\sigma(j)\leq i<j$. We let $\crs(\sigma)$ denote the number of crossings of $\sigma$.  For graphical understanding, we usually draw arc diagrams, i.e., draw an upper (resp., a lower) arc from $i$ to $\sigma(i)$ if $\sigma(i)>i$ (resp., $\sigma(i)<i$). 

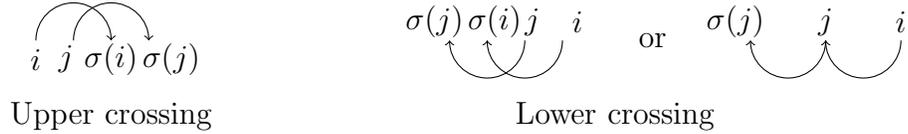
\begin{figure}[H]
	\centering
	\begin{tikzpicture}
	\draw[<-] (2.3,0) arc(0:180:0.5);
	\draw[<-] (2.8,0) arc(0:180:0.5);
	\draw (1.3,-0.25) node {$i$};
	\draw (2.3,-0.25) node {$\sigma(i)$};
	\draw (1.7,-0.25) node {$j$};
	\draw (3.1,-0.25) node {$\sigma(j)$};

	\draw (2.3,-1) node {Upper crossing};

	\draw[->] (7.8,0) arc(0:-180:0.5);
	\draw[->] (8.3,0) arc(0:-180:0.5);
	\draw (8.5,0.25) node {$i$};
	\draw (7.4,0.25) node {$\sigma(i)$};
	\draw (7.9,0.25) node {$j$};
	\draw (6.6,0.25) node {$\sigma(j)$};
	\draw (9.5 ,0) node {or};

	\draw[->] (11.8,0) arc(0:-180:0.5);
	\draw[->] (12.8,0) arc(0:-180:0.5);
	\draw (12.8,0.25) node {$i$};
	\draw (11.8,0.25) node {$j$};
	\draw (10.6,0.25) node {$\sigma(j)$};
	\draw (9,-1) node {Lower crossing};
	\end{tikzpicture}
	\caption{Arc diagrams  of crossings.}
	\label{fig:cross}
\end{figure}

Example: the crossings of the permutation $\pi=4735126\in S_7$ drawn in Figure \ref{fig:arcdiag} are $(1,2)$, $(5,6)$ and $(6,7)$. So we have $\crs(\pi)=3$.

\begin{figure}[H]
	\centering
	\begin{tikzpicture}
	\draw[black] (0,1) node {$1~2~3~4~5~6~7$}; 
	\draw (-0.9,1.2) parabola[parabola height=0.2cm,red] (-0,1.2);\draw[->,black] (-0.09,1.25)--(-0,1.2);			
	\draw (-0.6,1.2) parabola[parabola height=0.3cm,red] (0.85,1.2); \draw[->,black] (0.8,1.25)--(0.85,1.2);			
	\draw (-0.45,0.7) -- (-0.35,0.85)[rounded corners=0.1cm] -- (-0.25,0.7) -- cycle;			
	\draw (-0.03,1.2) parabola[parabola height=0.2cm,red] (0.3,1.2); \draw[->,black] (0.27,1.25)--(0.3,1.2);			
	\draw (-0.9,0.8) parabola[parabola height=-0.3cm,red] (0.3,0.8); \draw[->,black] (-0.85,0.75)--(-0.9,0.8);			
	\draw (-0.6,0.8) parabola[parabola height=-0.3cm,red] (0.6,0.8); \draw[->,black] (-0.55,0.75)--(-0.6,0.8);			
	\draw (0.6,0.8) parabola[parabola height=-0.1cm,red] (0.9,0.8); \draw[->,black] (0.65,0.75)--(0.6,0.8);			
	\end{tikzpicture}
	\caption{Arc diagrams of $\pi=4735126 \in S_7$ with $\crs(\pi)=3$.}
	\label{fig:arcdiag}
\end{figure}
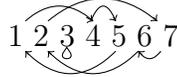

Let $\sigma\in S_n$ and $\tau \in S_k$ with $1\leq k\leq n$. For a given sequence of integers 
$$i_1<i_2<\cdots <i_k,$$ we say that a subsequence $s=\sigma(i_1)\sigma(i_2)\cdots \sigma(i_k)$ of $\sigma$ is an occurrence of $\tau$ if $s$ and $\tau$ are in order isomorphic, i.e.,
$\sigma(i_x)<\sigma(i_y)$ if and only if $\tau(x)<\tau(y)$. If there is no occurrence of the pattern $\tau$ in $\sigma$, we say that $\sigma$ is $\tau$-\textit{avoiding}. Example: the permutation $\pi=4162375 \in S_7$ is $321$-avoiding and it has five occurrences  of the pattern 312 namely 312, 423, 623, 625  and 635. We let $S_n(\tau)$ denote the set of all $\tau$-avoiding permutations of $[n]$. For a subset of patterns $T=\{\tau_1,\tau_2,\ldots\}$, we usually write $S_n(\tau_1,\tau_2,\ldots)$ for $S_n(T)$ and $S(T):=\cup_{n\geq 0}S_n(T)$. 
There are three useful trivial involutions on $S_n$ namely \textit{reverse} r, \textit{complement} c and \textit{inverse} i defined as follows:  	for a permutation $\sigma\in S_n$,
\begin{itemize}
	\setlength\itemsep{-0.3em}
	\item the reverse of $\sigma$ is r$(\sigma)=\sigma(n)\sigma(n-1)\cdots \sigma(1)$,
	\item the complement of $\sigma$ is c$(\sigma)=(n+1-\sigma(1))(n+1-\sigma(2))\cdots (n+1-\sigma(n))$,
	\item the inverse of $\sigma$ is i$(\sigma)=p(1) p(2) \cdots p(n)$ where $p(i)$ is the position of $i$ in $\sigma$. We often write i$(\sigma)=\sigma^{-1}$.
\end{itemize}	
Example: for $\pi=4135762 \in S_7$, we have r$(\pi)=2675314$, c$(\pi)=4753126$,  $\pi^{-1}=2731465$, ${\rm r\circ c}(\pi)=6213574$ and ${\rm r\circ c \circ i}(\pi)=3247516$ where $\circ$ denotes the composition operation.  Let
${\rm fg:=f\circ g}$ for an involution f and g in $\{{\rm r,c,i}\}$. By composition $\circ$, these defined involutions generate the dihedral group $\mathcal{D}=\{{\rm id,r,c,i,rc,ri,ci,rci}\}$ and they greatly simplify enumeration of pattern-avoiding permutations statistics through the fundamental property by Simion and Smith \cite{SiS} 
\begin{equation}
\varphi(S_n(T))= S_n(\varphi(T)) \text{ for $\varphi \in \mathcal{D}$ and a subset of patterns $T$}. \label{eqfunda}
\end{equation}

For a given statistic st, we say that two subsets $T_1$ and $T_2$ are st-\textit{Wilf-equivalent} if and only if the polynomial distributions of $\st$  over the sets $S_n(T_1)$ and $S_n(T_2)$ are the same for all integers $n$. In other word, for every integer $n$, we have
\begin{equation*}
\sum_{\sigma \in S_n(T_1)} x^{{\rm st}(\sigma)}=\sum_{\sigma\in S_n(T_2)} x^{{\rm st}(\sigma)}.
\end{equation*} 
Various statistic-Wilf-equivalence classes for subset of patterns of length three are known in \cite{BKLPRW, Dokos, Eliz1, Rakot, ARob}.  In particular, Rakotomamonjy \cite{Rakot} provided the Wilf-equivalence classes modulo $\crs$ for single pattern of length three. He proved bijectively that the only non singleton class is $\{132,213,321\}$, i.e., we have
\begin{equation}\label{eq:res1}
\sum_{\sigma\in S_n(321)} q^{\crs(\sigma)}=\sum_{\sigma\in S_n(132)} q^{\crs(\sigma)}=\sum_{\sigma\in S_n(213)} q^{\crs(\sigma)}.
\end{equation}
To prove the first identity of \eqref{eq:res1}, he exploited the bijection $\Theta:S_n(321)\rightarrow S_n(132)$ exhibited by Elizalde and Pak \cite{ElizP} and proved that $\Theta$ is \crs-preserving \cite[Thm. 3.10]{Rakot}. The second  identity of \eqref{eq:res1} is simply obtained from the fact that the  reverse-complement-inverse rci preserves the number of crossings \cite[Lem. 4.2]{Rakot}). Using the $q,p$-Catalan numbers defined by Randrianarivony \cite{ARandr}, Rakotomamonjy also proved the following result:
\begin{theorem}{\rm \cite{Rakot}}\label{thmRakot}
	Let $\tau\in \{321,132,213\}$. The polynomial $\displaystyle F_n(\tau;q):=\sum_{\sigma\in S_n(\tau)} q^{\crs(\sigma)}$  satisfies
	\begin{equation*}\label{eq:Randr}
	F_n(\tau;q)=F_{n-1}(\tau;q)+\sum_{k=0}^{n-2}q^kF_k(\tau;q)F_{n-1-k}(\tau;q).
	\end{equation*}
	Moreover, we have 
	\begin{equation*}
	\sum_{\sigma\in S(\tau)} q^{\crs(\sigma)}z^{|\sigma|}=\frac{1}{1-\displaystyle\frac{z}{				
			1-\displaystyle\frac{z}{
				1-\displaystyle\frac{qz}{								
					1-\displaystyle\frac{qz}{
						1-\displaystyle\frac{q^2z}{
							1-\displaystyle\frac{q^2z}{				
								\ddots}
						}
				}}		
	}}}. \label{eq:fc}
	\end{equation*}
\end{theorem}
Burstein and Elizalde found that this continued fraction expansion is the distribution of the statistic number of occurrences of the generalized pattern 31-2 in 231-avoiding permutations \cite[Thm.~3.11]{BEliz}. For interested reader, knowing that Corteel \cite{Cort} established the connection between occurrences of patterns, crossings and nestings on permutations, finding any correspondence between these results may be interesting. Notice also that finding the polynomial distributions of the number of crossings over the sets $S_n(\tau)$, for 
$\tau \in \{123,231,312\}$, remain open. The first result of this paper is the following.	
\begin{theorem} \label{thm1} 		
	We have the following identities:
	\begin{align}
	\sum_{\sigma \in S(231,321)}q^{\crs(\sigma)}z^{|\sigma|}&=\frac{1-qz}{1-(1+q)z-(1-q)z^2}, \label{main31}\\	
	\sum_{\sigma \in  S(123,\tau)}q^{\crs(\sigma)}z^{|\sigma|}&=1+\frac{(1-qz)z}{(1-z)(1-(1+q)z)} \text{ for } \tau \in \{132,213\}. \label{main32}
	\end{align}
\end{theorem}
We observe throughout the paper of Bukata et al.~\cite{BKLPRW} that  identities \eqref{main31} and \eqref{main32} are, respectively, new combinatorial interpretations of the triangles \seqnum{A076791} and \seqnum{A299927} of the On-Line Encyclopedia of Integer Sequences (OEIS) \cite{OEIS}. Bukata et al.~interpreted these triangles in terms of number of double descents (ddes) and number of double ascents(dasc) over permutations avoiding some pairs of patterns of length three \cite[Prop.~7 and Prop.~11]{BKLPRW}. The statistics ddes and dasc are,
respectively, defined by ${\rm ddes}(\sigma):=|\{i|\sigma(i)>\sigma(i+1)>\sigma(i+2)\}|$  and ${\rm dasc}(\sigma):=|\{i|\sigma(i)<\sigma(i+1)<\sigma(i+2)\}|$ for a permutation  $\sigma$. Notice that the triangle \seqnum{A299927} is new in the OEIS and it  was first discovered by Bukata et al.~.

Let $\tau \in \{132,213\}$. For an integer $n\geq 1$  and  $k\geq 0$, as direct consequence of identity \eqref{main32}, we have  
\begin{equation*}
|\{\sigma \in S_n(123,\tau)|\crs(\sigma)=k\}|=\delta_{k,0}+\binom{n-1}{k+1}.
\end{equation*}

The next result of this paper concerns various relationships between the distributions of the number of crossing over permutations that avoid the pattern 231 on the first hand and permutations that avoid the pattern 312 on the second hand.  For that, we let $F(T;q,z):=\sum_{\sigma \in S(T)}q^{\crs(\sigma)}z^{|\sigma|}$ for a subset of patterns $T$.
\begin{theorem} \label{thm2} We have the following identities:
	\begin{align*}
	F(312;q,z)&=\frac{1}{1-zF(231;q,z)}, \\
	F(312,123; q,z)&= 1+\left(\frac{z}{1-z} \right)^2 +zF(231,123; q,z), \\	
	\text{ and }	F(312,\tau; q,z)&= 1+\left( \frac{z}{1-z}\right)F(231,\tau'; q,z)  \text{ for } (\tau,\tau') \in   \{132,213\}^2. 
	\end{align*}
\end{theorem}

The aim of this paper is to find the polynomial distributions of the number of crossings over permutations avoiding a pair of patterns in $S_3$. The tool that we use is not sufficient to treat all cases. However, these relationships we found will obviously reduce the number of the remain cases to be processed.

We organize the rest of this paper in three sections. Section \ref{sec2} is for notation and preliminary in which we will prove one fundamental proposition that will play a central role in the proof of our results. In Section \ref{sec3}, we will provide the proof of our main results. In Section \ref{sec4}, we will end this paper with two additional results. The first one is about the distributions of the number of crossings over the sets $S_n(321,213)$ and $S_n(321,132)$. The second one is about a $\crs$-preserving bijection between $S_n^{k}$ and $S_n^{n+1-k}$.

\section{Notation and preliminary}\label{sec2}	

Let $n$ be a positive  integer. For $k \in [n]$, we write
\begin{align*}
S_{n}^{k} &:=\{\sigma \in S_n|\sigma(k)=1\} \\
S_{n,k} & :=\{\sigma \in S_n|\sigma(n)=k\}.
\end{align*}
We let $F_n(T;q)$, $F_n^k(T; q)$ and $F_{n,k}(T;q)$
denote, respectively, the polynomial distributions of $\crs$ over
the sets $S_n(T)$, $S_n^k(T)$ and $S_{n,k}(T)$, for any subset of
patterns $T$ and any integer $k\in [n]$.  In particular, we let
$F_n(q):=F_n(\emptyset;q)$, $F_n^k(q):=F_n^k(\emptyset; q)$ and
$F_{n,k}(q)=F_{n,k}(\emptyset;q)$.

Let $m$ and $n$ be two integers such that $m>1$. Let $T\subset S_m$ and $k\in [n]$. We also write $T^{-1}:=\{\tau^{-1}|\tau\in T \}$ and $T(i):=\{\tau(i)|\tau \in T\}$ for $i\in [m]$. In this section, we will prove the following fundamental proposition that will help us to solve our problems in the next sections.
\begin{proposition}\label{prop21}
	For all integer $n\geq 1$, the following properties hold:
	\begin{align}
	\text{If }& \min T^{-1}(1)>1, \text{ we have } F_{n}^{1}(T;q)=F_{n-1}(T;q).\label{stat1}\\
	\text{If }&  \min T^{-1}(1)>2, \text{ we have } F_{n}^{2}(T;q)= qF_{n-1}(T;q)+(1-q)F_{n-2}(T;q). \label{stat2} \\
	\text{If }& \max T^{-1}(1)<m-1, \text{ we have } F_{n}^{n-1}(T;q)=qF_{n-1}(T^{-1};q)+(1-q)F_{n-1,n-1}(T^{-1};q). \label{stat3} \\	
	\text{If } & \max T^{-1}(1)<m, \text{ we have }  F_{n}^{n}(T;q)=F_{n-1}(T^{-1};q). \label{stat4} 
	\end{align}
\end{proposition}
For that, we need some notation to be defined and some lemmas to be proved. So, we let $\sigma\in S_n$. We say that an integer $i$ is an upper transient (resp., lower transient) of $\sigma$ if and only if $\sigma^{-1}(i)<i<\sigma(i)$ (resp., $\sigma(i)<i<\sigma^{-1}(i)$). The numbers of upper and lower transients of a given permutation $\sigma$ are denoted, respectively, by $\ut(\sigma)$  \text{ and }	$\lt(\sigma)$. With this definition, we have the following remark.
\begin{remark}\label{rem0}
	For any permutation $\sigma$, an integer $i$ is a lower transient of $\sigma$ if  and only if $(i,\sigma^{-1}(i))$ is a lower crossing of $\sigma$.
\end{remark} For any given integer $k$, we also let
$\Ut_k(\sigma):=\{i<k/\sigma^{-1}(i)<i<\sigma(i)\}$ and $\Lt_k(\sigma):=\{i<k/\sigma(i)<i<\sigma^{-1}(i)\}$ denote, respectively,
the sets of upper and lower transients of $\sigma$ less than $k$. Also define
\begin{align*}
\ut_k^-(\sigma)&:=|\Ut_k(\sigma)| \text{ and } \ut_k^+(\sigma):=\ut(\sigma)-\ut_k^-(\sigma),\\
\lt_k^-(\sigma)&:=|\Lt_k(\sigma)| \text{ and } \lt_k^+(\sigma):=\lt(\sigma)-\lt_k^-(\sigma),\\
\alpha_k(\sigma)&:=|\{i\geq k/\sigma(i)<k\}|.
\end{align*}
Observe that in particular we have $\ut_n^-(\sigma)=\ut_{n+1}^-(\sigma)=\ut(\sigma)$ and $\lt_n^-(\sigma)=\lt_{n+1}^-(\sigma)=\lt(\sigma)$, $\alpha_{n}(\sigma)=1-\delta_{n,\sigma(n)}$ and $\alpha_{n+1}(\sigma)=0$ where $\delta$ is the usual Kronecker symbol. Now, let us recall  one needed notation introduced by Rakotomamonjy \cite{Rakot}. Given a permutation $\sigma$  and two integers $a$ and $b$, we let $\sigma^{(a,b)}$ denote the permutation obtained from $\sigma$ in the following way:
\begin{itemize}
	\setlength\itemsep{-0.3em}
	\item add $1$ to each number in $\sigma$ which is greater or equal to $b$,
	\item then insert $b$ at the $a$-th position of the modified $\sigma$.
\end{itemize}
We can simply write $\sigma^{-(a,b)}$ for $(\sigma^{-1})^{(a,b)}$. Example: we have $3142^{(2,\textcolor{gray}{3})}=4\textcolor{gray}{3}152$ and $3142^{-(2,\textcolor{gray}{3})}$\\$=2\textcolor{gray}{3}514$. Next, we prove a fundamental lemma which is a particular case of \cite[Lem. 3.6]{Rakot}. 
\begin{lemma}\label{lem21} 	Let $\sigma \in S_{n}$ and $k\in [n+1]$. We have \begin{equation*}
	\crs(\sigma^{(k,1)})=\crs(\sigma)+\ut_k^-(\sigma)-\lt_k^-(\sigma)+\alpha_k(\sigma).
	\end{equation*}
\end{lemma}
\begin{proof}
	Let $\sigma \in S_n$ and $k \in [n+1]$. Firstly,  we let
	$A_k(\sigma)$ (resp., $B_k(\sigma)$, $C_k(\sigma)$) denote the set of all crossings $(i,j)$ of $\sigma$ such that $j<k$ (resp., $i<k\leq j$, $k\leq i$). We obviously have $\crs(\sigma)=|A_k(\sigma)|+|B_k(\sigma)|+|C_k(\sigma)|$. 
	Let us assume that $\pi=\sigma^{(k,1)}$. By definition, we have 
	\begin{equation*}
	\pi(i)=\sigma(i)+1 \text{ if }  i<k,	\pi(k)=1 \text{ and } \pi(i+1)=\sigma(i)+1   \text{ if } i\geq k.
	\end{equation*}
	Let  $(i,j)$ be a pair of integers such that $i<j$. Based on this definition of $\pi$, we will examine the following three cases:
	\begin{description}
		\item[Case 1:] Suppose that $j<k$. So we have $\pi(i)=\sigma(i)+1$ and $\pi(j)=\sigma(j)+1$.
		\begin{itemize}
			\item Assume that $(i,j)\in A_k(\sigma)$.
			\begin{itemize}
				\item If $i<j<\sigma(i)<\sigma(j)$, then $i<j<\pi(i)<\pi(j)$ and $(i,j)\in A_k(\pi)$, \\[.1in]
				\item If $\sigma(i)<\sigma(j)\leq i<j$, then $\begin{cases}
				\pi(i)<\pi(j)\leq i<j,  &\text{ if } \sigma(j)<i;\\
				\pi(i)\leq i<\pi(j)=i+1\leq j, &\text{ if } \sigma(j)=i.
				\end{cases}$ 
				
				Thus, we have $\begin{cases}
				(i,j)\in A_k(\pi), &\text{ if } \sigma(j)<i;\\
				(i,j)\notin A_k(\pi),& \text{ if } \sigma(j)=i.
				\end{cases}$
			\end{itemize}
			\item Inversely, if $(i,j)\in A_k(\pi)$, the following properties hold:
			\begin{itemize}
				\item if $i<j<\pi(i)<\pi(j)$, then $i<j\leq \sigma(i)<\sigma(j)$. So,
				we have 
				\begin{equation*}
				\begin{cases}
				(i,j)\in A_k(\sigma), &\text{ if }  \pi(i)>j+1  (\text{i.e., }\sigma(i)>j);\\
				(i,j)\notin A_k(\sigma), &\text{ if } \pi(i)=j+1  (\text{i.e., } \sigma(i)=j).
				\end{cases}
				\end{equation*}
				\item if $\pi(i)<\pi(j)\leq i<j$ then $\sigma(i)<\sigma(j)< i<j$, i.e., $(i,j)\in A_k(\sigma)$.  
			\end{itemize}
		\end{itemize}
		Consequently, we obtain the following identity 
		\begingroup\small\begin{equation}\label{eq11}
		|A_k(\sigma)|-|\{i|\sigma(i)<i<\sigma^{-1}(i)<k\}|=|A_k(\pi)|-|\{(i,j)\in A_k(\pi)|i<j<\pi(i)=j+1\}|.
		\end{equation}\endgroup

		\item[Case 2:] Suppose that $i<k\leq j$. We have $\pi(i)=\sigma(i)+1$ and $\pi(j+1)=\sigma(j)+1$.
		\begin{itemize}
			\item Assume that $(i,j)\in B_k(\sigma)$.
			\begin{itemize}
				\setlength\itemsep{-0.2em}
				\item If $i<j<\sigma(i)<\sigma(j)$ then $(i,j+1)\in B_k(\pi)$,
				\item If $\sigma(i)<\sigma(j)\leq i<j$ then  $\begin{cases}
				(i,j+1)\in B_k(\pi), &\text{ if } \sigma(j)<i;\\
				(i,j+1)\notin B_k(\pi), & \text{ if } \sigma(j)=i.
				\end{cases}$
			\end{itemize}
			\item Inversely, if $(i,j)\in B_k(\pi)$,
			\begin{itemize}
				\setlength\itemsep{-0.2em}
				\item if $i<j<\pi(i)<\pi(j)$, then $j>k$ since $\pi(k)=1$. Thus, we have $i<j-1<\sigma(i)<\sigma(j-1)$, i.e., 	$(i,j-1)\in B_k(\sigma)$,
				\item if $\pi(i)<\pi(j)\leq i<j$, then $\sigma(i)<\sigma(j-1)< i<j-1$, i.e., $(i,j-1)\in B_k(\sigma)$.  
			\end{itemize}
		\end{itemize}
		Consequently, we obtain the following identity 
		\begin{equation}\label{eq12}
		|B_k(\sigma)|-|\{i|\sigma(i)<i<k\leq \sigma^{-1}(i)\}|=|B_k(\pi)|.
		\end{equation}
		\item[Case 3:] Suppose now that $k\leq i<j$. We have $\pi(i+1)=\sigma(i)+1$ and $\pi(j+1)=\sigma(j)+1$.
		\begin{itemize}
			\item If $(i,j)\in C_k(\sigma)$, then we have
			\begin{itemize}
				\item if $i<j<\sigma(i)<\sigma(j)$, then $(i+1,j+1)\in C_k(\pi)$,
				\item if $\sigma(i)<\sigma(j)\leq i<j$ then $(i+1,j+1)\in C_k(\pi)$.
			\end{itemize}
			\item Inversely, if $(i,j)\in C_k(\pi)$, we have 
			\begin{itemize}
				\item if $i<j<\pi(i)<\pi(j)$, then $k>i$ since $\pi(k)=1$. Thus, we have $k\leq i-1<j-1<\sigma(i-1)<\sigma(j-1)$, i.e., $(i-1,j-1)\in C_k(\sigma)$,
				\item if $\pi(i)<\pi(j)\leq i<j$, then $\sigma(i-1)<\sigma(j-1)\leq i-1<j-1$. So, we get
				$\begin{cases}
				(i-1,j-1)\in C_k(\sigma), & \text{ if } i>k;\\
				(i-1,j-1)\notin C_k(\sigma), & \text{ if } i=k.\\
				\end{cases}$  
			\end{itemize}
		\end{itemize}
		Similarly to the previous cases, we obtain 
		\begin{equation}\label{eq13}
		|C_k(\sigma)|=|C_k(\pi)|-|\{j>k|\pi(j)\leq k\}|.
		\end{equation}
	\end{description}
	By summing equations \eqref{eq11}, \eqref{eq12} and \eqref{eq13}, using the facts that $|\{i|\sigma(i)<i<\sigma^{-1}(i)<k\}|+|\{i|\sigma(i)<i<k\leq \sigma^{-1}(i)\}|=\lt_k^{-}(\sigma)$, $|\{(i,j)\in A_k(\pi)|i<j<\pi(i)=j+1\}|=|\{j<k|\sigma^{-1}(j)<j<\sigma(j)\}|=\ut_k^{-}(\sigma)$ and $|\{j>k|\pi(j)\leq k\}|=|\{j\geq k|\sigma(j)< k\}|=\alpha_k(\sigma)$, we get 
	\begin{equation}\label{eq14}
	\crs(\sigma)-\lt_k^{-}(\sigma)=\crs(\pi)-\ut_k^{-}(\sigma)-\alpha_k(\sigma).
	\end{equation}
	We deduce from \eqref{eq14} the desired identity of our lemma.
\end{proof}

\begin{lemma}\label{lem22}
	Let $\sigma$ be a given permutation. If $\pi=\sigma^{-1}$ or ${\rm rc}(\sigma)$ then we have \begin{equation*}
	\crs(\pi)=\crs(\sigma)+\ut(\sigma)-\lt(\sigma). 
	\end{equation*}
\end{lemma}
\begin{proof}
	Let $\sigma \in S_n$ and $\pi=\sigma^{-1}$ or ${\rm rc}(\sigma)$. Noticing 	that i or rc are simple symmetries on the arc diagram, they exchange lower and upper arcs including of course transients. Thus, we have $\ut(\pi)=\lt(\sigma)$ and $\lt(\pi)=\ut(\sigma)$. By this fact, Remark \ref{rem0} explains how we get $\crs(\pi)=\crs(\sigma)+\ut(\sigma)-\lt(\sigma)$ and we complete the proof of our lemma.
\end{proof}

Let $n$ be an integer and $k \in [n]$. Let us now define a bijection $\Phi_{n,k}$ as follows 
\begin{align*}
\Phi_{n,k}:S_{n-1} &\longrightarrow S_{n}^k\\
\sigma&\longmapsto   \sigma^{-(k,1)}. 
\end{align*}
The properties of this bijection allow us to get some relations between $F_n^{n-1}$, $F_n^{n}$ and $F_n$ in Proposition \ref{pro22}  and we use its restricted version  to prove Proposition \ref{prop21}. 
\begin{proposition}\label{prop22}
	The bijection $\Phi_{n,n}$ preserves the number of crossings and, for any $\sigma \in S_{n-1}$, we have
	\begin{equation*}
	\crs(\Phi_{n,n-1}(\sigma))=	
	\begin{cases}
	\crs(\sigma),  &    \text{ if } \sigma(n-1)=n-1; \\
	\crs(\sigma)+1,  &    \text{ if } \sigma(n-1)<n-1.
	\end{cases}.  
	\end{equation*} 	
\end{proposition}
\begin{proof}
	Combining Lemma \ref{lem21}	and  Lemma \ref{lem22}, it is not difficult to see that, for any $\sigma\in S_{n-1}$, we have 
	\begin{equation}\label{eq22}
	\crs(\sigma^{-(n,1)})=\crs(\sigma) \text{ and  } \crs(\sigma^{-(n-1,1)})=\crs(\sigma)+1-\delta_{n-1,\sigma(n-1)}. 
	\end{equation}
	The proposition comes from \eqref{eq22}. 
\end{proof}

Let $\alpha\oplus \beta$ denote the  \textit{direct sum} of the two given permutations $\alpha$ and $\beta$  defined as follows
\begin{equation*}
\alpha\oplus \beta(i)=\begin{cases} \alpha(i), & \text{ if } i\leq |\alpha|;\\
|\alpha|+\beta(i-|\alpha|), & \text{ if } i> |\alpha|.
\end{cases}
\end{equation*}
Example: $1432\oplus 4231=14328675$. An obvious property of the direct sum that  we need is   $\crs(\alpha\oplus \beta)=\crs(\alpha)+ \crs(\beta)$ for any permutations $\alpha$ and $\beta$. 

\begin{proposition} \label{pro22}Let $n$ be a non-negative integer. The following recurrences hold
	\begin{align*}
	F_{n}^n(q)&=F_{n-1}(q) \text{ for } n\geq 1, \\
	\text{ and } F_{n}^{n-1}(q)&=qF_{n-1}(q)+(1-q)F_{n-2}(q) \text{ for } n\geq 2.
	\end{align*}
\end{proposition}
\begin{proof} Since the bijection $\Phi_{n,n}$ is \crs-preserving, we have $F_{n}^n(q)=F_{n-1}(q)$.
	Now, using the property of the bijection $\Phi_{n,n-1}$, we get
	\begin{align*}
	F_{n}^{n-1}(q)&=q\times \sum_{\substack{\sigma \in S_{n-1},\\ \sigma(n-1)\neq n-1}} q^{\crs(\sigma)}+\sum_{\substack{\sigma \in S_{n-1},\\ \sigma(n-1)= n-1}} q^{\crs(\sigma)}\\
	&=q\left(  F_{n-1}(q)-F_{n-1,n-1}(q)\right) +F_{n-1,n-1}(q)
	\end{align*}	
	Since 
	$\displaystyle F_{n,n}(q)=\sum_{\sigma\oplus 1 \in S_{n}} q^{\crs(\sigma\oplus 1)}=\sum_{\sigma \in S_{n-1}} q^{\crs(\sigma)}=F_{n-1}(q)$ for all $n\geq 1$, we consequently obtain 
	\begin{equation*}
	F_{n}^{n-1}(q)=qF_{n-1}(q)+(1-q)F_{n-2}(q) \text{ for all $n\geq 1$}.
	\end{equation*}	
	This  ends the proof of the proposition.
\end{proof}	
We may observe that Proposition \ref{prop21} is  none other than a restricted version of Proposition \ref{pro22}. In fact, the effect of restriction totally changes the obtained relations. For example, we have $F_{n}^{n}(321; q)= 1\neq F_{n-1}(321; q)$.  Now, we can give the poof of Proposition \ref{prop21}. 
\begin{proof}
	Our proof is simply based on the following obvious fact. Let $T$ be a subset of $S_m$ for any integer $m>1$. For any integer $n\geq m$, we have
	\begin{itemize}
		\setlength\itemsep{-0.3em}
		\item[(i)] If $k<\min T^{-1}(1)$, we have $\sigma^{(k,1)} \in S_n^k(T)$ if and only if $\sigma \in S_{n-1}(T)$.
		\item[(ii)] If $n-m+\max T^{-1}(1)<k\leq n$, we have $\sigma^{-(k,1)} \in S_n^k(T)$ if and only if $\sigma \in S_{n-1}(T^{-1})$.
	\end{itemize}
	The two first relations \eqref{stat1} and \eqref{stat2} of Proposition \ref{prop21} use the (i) of the fact. If $\min T^{-1}(1)\neq 1$, then  we  have $1\oplus \sigma \in S_n^1(T)$ if and only if $\sigma \in S_{n-1}(T)$ for any  $n\geq 1$. Thus we get relation \eqref{stat1} as follows
	\begin{equation*}
	F_n^1(T;q)=\sum_{1\oplus \sigma \in S_{n}^1(T)} q^{\crs(1\oplus \sigma)}=\sum_{\sigma \in S_{n-1}(T)} q^{\crs(\sigma)}=F_{n-1}(T;q).
	\end{equation*}	
	By the same way, if $\min T^{-1}(1)>2$, we  have $\sigma^{(2,1)} \in S_n^2(T)$ if and only if $\sigma \in S_{n-1}(T)$  for any  $n\geq 1$. Moreover, we have $\crs(\sigma^{(2,1)})=\crs(\sigma)+1-\delta_{1,\sigma(1)}$ for any permutation $\sigma$ (see Lemma \ref{lem21}). By applying \eqref{stat1}, we also get \eqref{stat2} as follows
	\begin{align*}
	F_{n}^{2}(T;q) &= q\times \sum_{\substack{\sigma \in S_{n-1}(T),\\ \sigma(1)\neq 1}} q^{\crs(\sigma)}+\sum_{\substack{\sigma \in S_{n-1}(T),\\ \sigma(1)= 1}}q^{\crs(\sigma)}\\
	&=q\left( F_{n-1}(T;q)-F_{n-1}^1(T;q)\right) +F_{n-1}^1(T;q)\\
	&=qF_{n-1}(T;q)+(1-q)F_{n-2}(T;q).
	\end{align*}	
	For the two last relations \eqref{stat3} and \eqref{stat4} of the proposition, we obviously use the (ii)  of the fact and we also exploit the bijections $\Phi_{n,n}$ and $\Phi_{n,n-1}$. 
	If $\max T^{-1}(1)<m-1$ (i.e., $n-m+\max T^{-1}(1)<n-1$), we have $\sigma^{-(n-1,1)} \in S_n^{n-1}(T)$ if and only if $\sigma \in S_{n-1}(T^{-1})$.  This implies that we have  $\Phi_{n,n-1}(S_{n-1}(T^{-1}))=S_n^{n-1}(T)$. Using the property of the bijection $\Phi_{n,n-1}$ described in Theorem \ref{prop22}, we get \eqref{stat3} as follows
	\begin{align*}
	F_{n}^{n-1}(T;q) &= q\times \sum_{\substack{\sigma \in S_{n-1}(T^{-1}),\\ \sigma(n-1)\neq n-1}} q^{\crs(\sigma)}+\sum_{\substack{\sigma \in S_{n-1}(T^{-1}),\\ \sigma(n-1)= n-1}}q^{\crs(\sigma)}\\
	&=q\left( F_{n-1}(T^{-1};q)-F_{n-1,n-1}(T^{-1};q)\right) +F_{n-1,n-1}(T^{-1};q)\\
	&=qF_{n-1}(T^{-1};q)+(1-q)F_{n-1,n-1}(T^{-1};q).
	\end{align*}
	Notice that we generally have $F_{n,n}(T;q)\neq F_{n-1}(T;q)$ since the set $S_{n,n}(T)$ depends on $T$. By  the same way we obtain the last relation \eqref{stat4} using the \crs-preserving of the bijection $\Phi_{n,n}$. This ends the proof of Proposition \ref{prop21}.  
\end{proof}
Let us end this preliminary section with some illustrations of Proposition \ref{prop21}. Since $\max\{321\}^{-1}(1)=3>2$, by applying \eqref{stat2}
we get
\begin{equation*}
F_{n}^{2}(321;q)=qF_{n-1}(321;q)+(1-q)F_{n-2}(321;q) \text{ for } n\geq 2.
\end{equation*}
Since $123^{-1}=123$ and $\max\{123\}^{-1}(1)=1<2$, we can also apply \eqref{stat3} and get
\begin{equation*}
F_{n}^{n-1}(123;q)=qF_{n-1}(123;q)+(1-q)F_{n-1,n-1}(123;q).
\end{equation*}
Since $S_{n,n}(123)=\{(n-1)\cdots 21n\}$, then  we have $F_{n,n}(123;q)=1$ and  we consequently	obtain
\begin{equation*}
F_{n}^{n-1}(123;q)=qF_{n-1}(123;q)+1-q.
\end{equation*}

\section{Proof of the main results}\label{sec3}
In this section, we will establish the proof of our results presented in Section \ref{sec1}. As fundamental tools,  we use the Proposition \ref{prop21} proved in the preceding section and the $\crs$-preserving of the involution $\rci$ proved by Rakotomamonjy \cite{Rakot}. For that, we let $F(T;q,z):=\sum_{\sigma\in S(T)}q^{\crs(\sigma)}z^{|\sigma|}$ for any set of patterns $T$.

\subsection{Proof of Theorem \ref{thm1}}\label{sec31}

\begin{proof}
	It is obvious to see that we have $S_n(321,231)=S_n^1(321,231)\cup S_n^2(321,231)$ for all $n$. So we get
	\begin{equation*}
	F_n(321,231;q)=F_n^1(321,231;q)+F_n^2(321,231;q).
	\end{equation*}
	Since $\min\{321,231\}^{-1}(1)=3>2$, we can apply the relations \eqref{stat1} and \eqref{stat2} of proposition \ref{prop21}
	and we get 
	\begin{equation}\label{rec31}
	F_n(321,231;q)=(1+q)F_{n-1}(321,231;q)+(1-q)F_{n-2}(321,231;q), \text{ for $n\geq 2$}.
	\end{equation}
	Recurrence \eqref{rec31} is associated with the following functional equation  
	\begin{equation*}
	F(321,231; q,z)=1+z+(1+q)z(F(321,231; q,z)-1)+(1-q)z^2F(321,231; q,z).
	\end{equation*}
	Solving it by $F(321,231; q,z)$, we obtain the following identity equivalent to identity \eqref{main31} of Theorem \ref{thm1}:
	\begin{equation*}
	F(321,231; q,z)=\frac{1-qz}{1-(1+q)z-(1-q)z^2}.
	\end{equation*}
	
	As structure, we  also have $S_n(123,132)=S_n^{n-1}(123,132)\cup S_n^{n}(123,132)$. Since we have $\max\{123,132\}^{-1}(1)=1<2$, we can also apply the two relations \eqref{stat3} and \eqref{stat4} of Proposition \ref{prop21}. Thus, since $\{123,132\}^{-1}=\{123,132\}$,	we get   from \eqref{stat3}
	\begin{equation}
	F_n^n(123,132;q)=F_{n-1}(123,132;q). \label{eq:proof31}
	\end{equation}
	Moreover, since $F_{n,n}(123,132;q)=1$, we get from \eqref{stat4}
	\begin{equation} 
	F_n^{n-1}(123,132;q)=qF_{n-1}(123,132;q)+1-q. \label{eq:proof32}
	\end{equation}
	Summing \eqref{eq:proof31} and \eqref{eq:proof32}, we obtain the following recurrence:
	\begin{equation}\label{eq:proof33}
	F_{n}(123,132;q)=(1+q)F_{n-1}(123,132;q)+1-q \text{ for } n\geq 2.
	\end{equation}
	Recurrence \eqref{eq:proof33} corresponds to the following functional equation:
	\begin{equation*}
	F(123,132;q,z)=1+z+(1+q)z(F(123,132;q)-1)+z\left( \frac{1}{1-z}-1-z\right).
	\end{equation*} 
	When solving this last equation by $F(123,132;q,z)$, we obtain 
	\begin{equation*}
	F(123,132;q,z)=1+\frac{z(1-qz)}{(1-z)(1-(1+q)z)}.
	\end{equation*}
	Finally, since $\{123,213\}=\rci(\{123,132\})$, we also have $F(123,132;q,z)=F(123,213;q,z)$.
	This completes the proof of identity \eqref{main32} of Theorem \ref{thm1} and Theorem \ref{thm1} itself.
\end{proof}

Notice that when we solve  \eqref{eq:proof33} with the initial condition $F_{1}(123,132;q)=1$, we obtain the following closed form:
\begin{equation*}
\sum_{\sigma\in S_n(123,\tau)}q^{\crs(\sigma)}=\frac{(1+q)^{n-1}-1+q}{q}  \text{ for $n\geq 1$ and $\tau \in \{132,213\}$}.
\end{equation*}
Furthermore, when we substitute  $F_{n-1}(123,132;q)$ by $\frac{(1+q)^{n-2}-1+q}{q}$ for $n\geq 2$, we also get from \eqref{eq:proof32}
\begin{equation*}
\sum_{\sigma\in S_n^{n-1}(123,132)}q^{\crs(\sigma)}=(1+q)^{n-2}  \text{ for $n\geq 2$}.
\end{equation*}
Since $\rci(S_n^{n-1}(123,132))=S_{n,2}(123,213)$, we also get
\begin{equation*}
\sum_{\sigma\in S_{n,2}(123,213)}q^{\crs(\sigma)}=(1+q)^{n-2}  \text{ for $n\geq 2$}.
\end{equation*}
\begin{corollary}\label{cor32} For $n\geq 2$  and  $k\geq 0$, we have 
	\begin{equation*}
	|\{\sigma \in S_n^{n-1}(123,132)|\crs(\sigma)=k\}|=|\{\sigma \in S_{n,2}(123,213)|\crs(\sigma)=k\}|=\binom{n-2}{k}. 
	\end{equation*} 
\end{corollary}
We observe that Corollary \ref{cor32} is a new combinatorial interpretation of the Pascal triangle \seqnum{A007318}  in terms of crossings over restricted permutations.  Finding  a bijection with subsets of a given size to get a direct proof of Corollary \ref{cor32} may be interesting and staying open.

\subsection{Proof of Theorem \ref{thm2}} \label{sec32}
In this subsection, we will establish the proof of the result concerning some relationships between the distributions of crossings over the sets $S_n(312,T)$ and $S_n(231,T)$, where $T$ is empty or a singleton of $\{123,132,213\}$. As we did in the preceding subsection, we will first find recurrences and we then compute the corresponding generating functions to get the desired relations.

\begin{proposition}\label{prop33}
	For all integer $n\geq 1$, we have 
	\begin{equation}\label{rec1}
	F_n(312;q)=\sum_{j=0}^{n-1}F_{j}(231;q)F_{n-1-j}(312;q).
	\end{equation} 
\end{proposition}
\begin{proof}
	We have $S_n^j(312)=\{\sigma_1\oplus \sigma_2|\sigma_1\in S_j^j(312),\sigma_2\in S_{n-j}(312)\}$ for all $j\geq 1$. So, we get using  \eqref{stat4}  the following identities:
	\begin{equation*}
	F_n^j(312;q)=F_{j}^j(312;q) F_{n-j}(312;q)= F_{j-1}(231;q) F_{n-j}(312;q), \text{ for $1\leq j\leq n$}.
	\end{equation*}  
	By summing $F_n^j(312;q)$ over $j\in [n]$, we obtain the desired relationship for $F_n(312;q)$.
\end{proof}

\begin{proposition} \label{prop32} For all integer $n\geq 2$, we have	 
	\begin{equation}\label{rec3}
	F_n(123,312;q)=n-1+F_{n-1}(123,231;q).
	\end{equation}	
\end{proposition}
\begin{proof}
	We have $S_n(123,312)=\{\pi_1,\pi_2,\ldots, \pi_{n-1}\}\cup S_n^n(123,312)$ with $\pi_j=j\cdots 21n(n-1)\cdots(j+1)$ for all $j\in [n]$. So we get
	\begin{equation*}
	F_n(123,312;q)=\sum_{j=1}^{n-1}q^{\crs(\pi_j)}+F_n^n(123,312;q).
	\end{equation*}
	It is not difficult to see that we have $\crs(\pi_j)=0$ for all $j\in [n]$. Thus, we immediately obtain the proposition using again \eqref{stat4}.
\end{proof}

\begin{proposition} \label{prop34} For any $\tau_1,\tau_2$ and $\tau_3 \in \{132,213\}$ and for all $n\geq 2$, we have	 
	\begin{equation}\label{rec4}
	F_n(312,\tau_1;q)=F_{n-1}(312,\tau_2;q)+F_{n-1}(231,\tau_3;q).
	\end{equation}
\end{proposition}
\begin{proof}
	Since $S_n(312,213)=S_{n}^1(312,213)\cup S_n^{n}(312,213)$, we get
	\begin{equation*}
	F_n(312,213;q)=F_{n}^{1}(312,213;q)+F_{n}^{n}(231,213;q).
	\end{equation*}
	So  for all  $n\geq 2$, we get from \eqref{stat1} and \eqref{stat2} the following identity:
	\begin{equation*}
	F_n(312,213;q)=F_{n-1}(312,213;q)+F_{n-1}(231,213;q).
	\end{equation*}
	To complete the proof of the proposition, we just use the facts that $\rci(\{312,132\})= \{312,213\}$ and $\rci(\{231,132\})= \{231,213\}$.
\end{proof}
Now, to prove Theorem \ref{thm2}, we just compute the corresponding generating functions of the three recurrences \eqref{rec1},  \eqref{rec3}  and \eqref{rec4} and deduce the desired relations. 

\begin{proof}
	From \eqref{rec1}, we obtain the functional equation  
	\begin{equation*}
	F(312; q,z)=1+zF(312; q,z).F(231; q,z) 
	\end{equation*}
	which leads to 
	\begin{equation*}
	F(312;q,z)=\frac{1}{1-zF(231;q,z)}.
	\end{equation*}
	
	The associated generating function with \eqref{rec3} is \begin{equation*}
	F(123,312;q,z)=1+z+\left(\frac{z}{1-z}\right)^2 + z (F(123,231; q,z)-1).
	\end{equation*} This functional equation is equivalent to the following one:
	\begin{equation*}
	F(312,123; q,z) 1+\left(\frac{z}{1-z} \right)^2 +zF(231,123; q,z).
	\end{equation*}
	
	From \eqref{rec4}, when we set $\tau=\tau_1=\tau_2$ and $\tau'=\tau_3$, we get the functional equation \begin{equation*}
	F(312,\tau;q,z)=1+z+z\left( F(312,\tau;q,z)+F(231,\tau';q,z)-2\right).
	\end{equation*}
	Solving it for $F(312,\tau; q,z)$, we obtain 
	\begin{equation*}	F(312,\tau; q,z)= 1+\left( \frac{z}{1-z}\right)F(231,\tau'; q,z)  \text{ for any } (\tau,\tau') \in   \{132,213\}^2. 	
	\end{equation*} 
	This completes the proof of Theorem \ref{thm2}.
\end{proof}

\section{Additional results}	\label{sec4}
We end this paper with two additional results. The first one is about $F_n(321,\tau;q)$, with $\tau\in \{213,132\}$. The second one is inspired from the first section and is about a $\crs$-preserving bijection between $S_{n}^k$ and $S_{n}^{n+1-k}$.

For the first result, we remark that the distribution of $\crs$  over the set of permutations avoiding one of the pairs $\{321,213\}$ and $\{321,132\}$ can be computed. One of the tools that we may use is an interesting relationship proved by Randrianarivony \cite{ARandr}. He showed how  the statistic $\crs$ is related to other usual statistics through the following identity:
\begin{equation}\label {eqfin}
\crs(\sigma)=\inv(\sigma)-\exc(\sigma)-2\nes(\sigma),
\end{equation}
where, for any permutation $\sigma$, $\inv(\sigma):=|\{(i,j)|i<j \text{ and } \sigma(i)>\sigma(j) \}|$, $\exc(\sigma):=|\{i| \sigma(i)>i\}|$ and $\nes(\sigma):=|\{(i,j)|i<j<\sigma(j)<\sigma(i) \text{ or } \sigma(j)<\sigma(i)\leq i<j\}|$ are respectively the numbers of inversions, excedances and nestings of $\sigma$. Below is an unexpected result in which we try to use identity \eqref{eqfin} to get the proof. 

\begin{theorem} \label{thm41} Let $[n]_q=1+q+\cdots+q^{n-1}$  for any integer $n\geq 1$. For any $\tau \in\{132,213\}$, we have 
	\begin{equation*}
	\sum_{\sigma\in S_n(321,\tau)}q^{\crs(\sigma)}=1+\displaystyle\sum_{k=1}^{n-1}[n-k]_{q^k}.
	\end{equation*} 
\end{theorem}
\begin{proof} 
	It is easy to see that we have 
	$S_n(321,213)=S_n^1(321,213)\cup\{\alpha_2,\alpha_3,\ldots,\alpha_n\}$ where $\alpha_j=(n-j+2)\cdots (n-1)n12\cdots (n+1-j)$  for all $j\in [n]$.  From this structure, we get	
	\begin{equation*}
	F_n(321,213;q) =F_{n}^1(321,213;q)+\sum_{j=2}^{n}q^{\crs(\alpha_j)}.
	\end{equation*}
	Since every 321-avoiding permutations are nonesting \cite[Lem.~5.1]{Rakot}, we have 
	\begin{equation*}
	\crs(\alpha_j)=\inv(\alpha_j)-\exc(\alpha_j)=(j-1)(n-j) \text{ for all $j$}.
	\end{equation*} 
	Using the fact that $F_{n}^1(321,213;q)=F_{n-1}(321,213;q)$,  we obtain	
	\begin{equation*}
	F_n(321,213;q) = F_{n-1}(321,213;q)+\sum_{j=2}^{n}q^{(j-1)(n-j)}.
	\end{equation*}
	When we solve this recurrence with the initial condition $F_1(321,213;q)=1$, we obtain \begin{equation*}
	F_n(321,213;q)=1+\sum_{k=1}^{n-1}\sum_{j=1}^{k}q^{j(k-j)}=1+\sum_{k=1}^{n-1}[n-k]_{q^k}.
	\end{equation*}
	From the fact that  $F_n(321,213;q)=F_n(321,132;q)$ since $\{321,132\}=\rci(\{321,213\})$, we complete the proof of the theorem.
\end{proof}

For the second additional result, we notice first that we have $S_n^k=\{\sigma^{(k,1)}|\sigma \in S_{n-1} \}$. We will show that the following well-defined and bijective map preserves the number of crossings:
\begin{align*}
\Psi_{n,k}:S_{n}^{k}  &\longrightarrow S_{n}^{n+1-k}\\
\sigma^{(k,1)}&\longmapsto   {\rm rc}(\sigma)^{(n+1-k,1)}. 
\end{align*}
\begin{theorem}\label{thm21}
	The bijection $\Psi_{n,k}$ preserves the number of crossings for $1\leq k \leq n$.
\end{theorem}
\begin{proof}
	Let $\sigma^{(k,1)}\in S_n^{k}$ and $\pi^{(n+1-k,1)}=\Psi_{n,k}(\sigma^{(k,1)})$ for $\sigma \in S_{n-1}$. Knowing that {\rm rc} exchanges lower and upper arcs, it is not difficult to see that we have 
	\begin{equation} \label{eq26}
	\ut_{n+1-k}^-(\pi)=\lt_{k}^+(\sigma)  \text{ and  }   \lt_{n+1-k}^-(\pi)=\ut_{k}^+(\sigma).
	\end{equation} 
	Moreover, since  $|\{i<k/\sigma(i)\geq k\}|=|\{i\geq k/\sigma(i)<k\}|$, we get 
	\begin{equation}\label{eq27}
	\alpha_{n+1-k}(\pi)=\alpha_{k}(\sigma).
	\end{equation}
	Indeed, we have
	\begin{align*}
	\alpha_{n+1-k}(\pi)&=|\{n-i\geq n+1-k/\pi(n-i)<n+1-k\}|,\\
	&=|\{i\leq k-1/n-\sigma(i)<n+1-k\}|,\\
	&=|\{i<k/\sigma(i)>k-1\}|,\\
	&=|\{i<k/\sigma(i)\geq k\}|,\\
	&=\alpha_{k}(\sigma).
	\end{align*}
	Consequently, combining \eqref{eq26} and \eqref{eq27} with Lemma \ref{lem21} and Lemma \ref{lem22}, we get
	\begin{align*}
	\crs(\pi^{(n+1-k,1)})&=\crs(\pi)+\ut_{n+1-k}^-(\pi)-\lt_{n+1-k}^-(\pi)+\alpha_{n+1-k}(\pi),\\		
	&=\crs(\sigma)+\ut(\sigma)-\lt(\sigma)+\lt_{k}^+(\sigma)-\ut_{k}^+(\sigma)+\alpha_{k}(\sigma), \\
	&=\crs(\sigma)+\left( \ut(\sigma)-\ut_{k}^+(\sigma)\right) -\left( \lt(\sigma)-\lt_{k}^+(\sigma)\right) +\alpha_{k}(\sigma),\\
	&=\crs(\sigma)+\ut_{k}^-(\sigma)-\lt_{k}^-(\sigma)+\alpha_{k}(\sigma),\\
	&=\crs(\sigma^{(k,1)}).
	\end{align*}
	This proves the $\crs$-preserving of the bijection $\Psi_{n,k}$ and also ends the proof of Theorem \ref{thm21}.
\end{proof} 
\begin{corollary} \label{cor21} For any integers $n$ and $k \in [n]$, we have the following equidistributions:
	\begin{equation*}
	\sum_{\sigma\in  S_{n,k}} q^{\crs(\sigma)}=\sum_{\sigma\in  S_{n}^{n+1-k}} q^{\crs(\sigma)}=\sum_{\sigma\in  S_{n}^{k}} q^{\crs(\sigma)}=\sum_{\sigma\in  S_{n,n+1-k}} q^{\crs(\sigma)}.
	\end{equation*}
\end{corollary}
\begin{proof}
	We have  $S_{n}^{n+1-k}=\Psi_{n,k}(S_n^k)$ and $S_{n,n+1-k}={\rm rci}(S_n^k)$ for any $k \in [n]$. So we get these identities from the facts that the bijections $\Psi_{n,k}$ and $\rci$ are \crs-preserving.
\end{proof}
\begin{corollary}
	The number of permutations of $[2n]$ having $r$ crossings is always even for all integers $n\geq 1$ and $r\geq 0$.
\end{corollary}
\begin{proof}
	The number of permutations of $[2n]$ having $r$ crossings is $[q^{r}]F_{2n}(q)$ (i.e., the coefficient of the polynomial $F_{2n}(q)$), where  $F_{2n}(q)=\sum_{k=1}^{2n}F_{2n}^k(q)=2\sum_{k=1}^{n}F_{2n}^k(q)$.
\end{proof}

\section{Acknowledgment}
We  highly appreciate the comments and suggestions of the anonymous referees,
which significantly contributed to
improving the quality of the publication.

\end{document}